\documentclass[11pt]{amsart} 
\usepackage{verbatim, latexsym, amssymb, amsmath,color}
\usepackage{epsfig}
\usepackage{tikz}
\usepackage{hyperref}
\usepackage{float}
\usetikzlibrary{matrix}
\usepackage{graphicx}

\def\R{\mathbb R}

\def\Z{\mathbb Z}

\newtheorem{thm}{Theorem}[section]
\newtheorem{lemma}[thm]{Lemma}

\newtheorem{cor}[thm]{Corollary}
\newtheorem{prop}[thm]{Proposition}
\theoremstyle{remark}

\theoremstyle{definition}
\newtheorem{defi}[thm]{Definition}

\title[Free-boundary minimal surfaces in the 3d Schwarzschild]{On free-boundary minimal surfaces in the Riemannian Schwarzschild manifold}
\author{Rafael Montezuma}

\thanks{\nonumber \\ Mathematics Department, Princeton University \\ Address: Fine Hall, Washington Road, Princeton NJ 08544-1000 USA \\ Contact: rcabral@math.princeton.edu or +1(609)258-4233}

\begin{document}

\begin{abstract}
{ {Is it possible to obtain unbounded minimal surfaces in certain asymptotically flat 3-manifolds as a limit of solutions to a natural mountain pass problem with diverging boundaries? In this work, we give evidence that this might be true by analyzing related aspects in the case of the exact Riemannian Schwarzschild manifold. 

More precisely, we observe that the simplest minimal surface in this space has Morse index one. We prove also a relationship between the length of the boundary and the density at infinity of general minimal surfaces satisfying a free-boundary condition along the horizon.}}
\end{abstract}
\maketitle
\setcounter{tocdepth}{1}

\section{Introduction}\label{introduction}
In this paper we study analytic and geometric properties of minimal surfaces in the three-dimensional Riemannian Schwarzschild manifold that meet the horizon orthogonally. This space is a rotationally symmetric space-like slice of a exact solution to the Einstein field equations, which is a fundamental example in the study of static black holes. The study of minimal surfaces in spaces that are relevant in general relativity is motivated by the crucial role that such objects play in the proof of the positive mass theorem by Schoen and Yau \cite{SchYau}. 

In this article we compute the Morse index of a rotationally symmetric totally geodesic section of the Schwarzschild metric with respect to variations that are tangential along the horizon. We derive also a monotonicity formula for general free-boundary minimal surfaces in this space, and apply it to obtain a relationship between its boundary length and density at infinity. More precisely, we consider, for each $m\geq 0$, the three-dimensional domain
\begin{equation*}
M = \{x = (x_1, x_2, x_3) \in \R^3 : r \geq m/2\},
\end{equation*} 
endowed with the Riemannian metric 
\begin{equation*}
g = \bigg( 1+\frac{m}{2r}\bigg)^4 \delta,
\end{equation*}
where $r =r(x) = |x|$ represents the Euclidean distance from $x\in M$ to the origin, and $\delta$ denotes the Euclidean flat metric on $M$. In this space, we consider the totally geodesic Euclidean coordinate plane through $\{0\}$, i.e.
\begin{equation*}
\Sigma_0 = \{x \in M : x_3 =0\}. 
\end{equation*} 
Since the metric $g$ is conformal to the Euclidean, the surface $\Sigma_0$ meets the horizon $\{x \in M: r = m/2\} = \partial M$  orthogonally along $\partial \Sigma_0$. 

In the theory of black holes in general relativity, the parameter $m$ used in the preceding paragraph to introduce our spaces is called the ADM-mass, see \cite{SchYau}. Therefore, $\Sigma_0$ is a properly embedded free-boundary minimal surface of the Schwarzschild manifold of ADM-mass $m$, for all $m>0$. More precisely, the term properly embedded says, in particular, that the boundary $\partial \Sigma_0$ coincides with $\Sigma_0\cap \partial M$, while the terms minimal surface and free-boundary mean that its mean-curvature with respect to $g$ vanishes and that $\Sigma_0$ meets $\partial M$ orthogonally, respectively. Note that $\Sigma_0$ is simply a flat plane if $m=0$.

Throughout this work, the surface $\Sigma_0$ and any other surface obtained from this one by a space rotation will be called a plane through the origin, even though they are topological annuli.



One of the purposes of this paper is the computation of the Morse index of planes through the origin; i.e. the maximum number of directions in which the surface can be deformed in such a way that its area is decreased. 

We obtained the following result.

\begin{thm}\label{thm-A}
The Morse index of a plane through the origin is one.
\end{thm}

In the process of proving the above statement, we determine the maximal annular domain containing the horizon on which a plane through the origin is stable with respect to variations that are tangential on the boundary of $M$ and vanish along the outer circular component.

\begin{thm}\label{cor-max-stab}
The maximal annular domain of stability 
\begin{equation*}
\Sigma_0(R) = \Sigma_0 \cap \{|x|\leq R\}
\end{equation*}
of the plane through the origin, $\Sigma_0$, is that for which the value of $R$ satisfies 
\begin{equation*}
\log \sqrt{\frac{2R}{m}} = \frac{2R+m}{2R-m}.
\end{equation*}
\end{thm}

The approximate value of $R$ is $5.5$ times $m$. The proofs of these results rely on a ODE analysis. Similar methods were applied recently by Smith and Zhou \cite{SmiZhou} in their computation of the Morse index of the free-boundary critical catenoid. Proofs of that result were independently obtained by Tran \cite{Tran} and Devyver \cite{Dev}, and generalized to higher dimensions in \cite{SmStTrZh}.

Next, we state a result that relates the length of the boundary and the density at infinity of general free-boundary minimal surfaces $\Sigma$. The density at infinity is a number that represents the asymptotic behavior at infinity of the ratio between the area of the surface inside a given ball centered at the horizon, the set of points at distance at most a certain constant from $\partial M$, and the area of the portion of a plane through the origin in the same ball. This is a natural generalization of the notion of density at infinity for minimal submanifolds in Euclidean space. The free-boundary condition implies that the ratio above approaches the length of $\partial \Sigma$ divided by the length of a great circle in the horizon, as the ball degenerates to the horizon. Using a monotonicity argument, we obtain:


\begin{thm}\label{thm-length-bound}
Let $\Sigma$ be a properly embedded minimal surface in $M$ that meets $\partial M$ orthogonally. Then, the length of the boundary of $\Sigma$ satisfies
\begin{equation*}
|\partial \Sigma| \leq 4\pi m \Theta (\Sigma).
\end{equation*}
Moreover, if equality holds $\Sigma$ is a plane through the origin and $\Theta (\Sigma)=1$.
\end{thm}


One of the major arguments in Schoen and Yau's proof of the positive mass theorem is a non-existence result of stable minimal surfaces arising from a limit of Plateau solutions with certain diverging circular boundaries in a space with positive scalar curvature. More recent developments show that asymptotically flat three-manifolds with positive scalar curvature or with non-negative scalar curvature and certain Schwarzschild asymptotics contain no embedded unbounded stable minimal surfaces; for more details see the works of Carlotto \cite{Car} and Carlotto, Chodosh, and Eichmair \cite{CarChoEic}. 

On the other hand, Chodosh and Ketover \cite{ChoKet} proved existence of many complete properly embedded minimal planes in asymptotically flat three-manifolds containing no closed embedded minimal surfaces. This was later generalized by Mazet and Rosenberg \cite{MazRos}. In the work \cite{ChoKet}, the authors apply degree theoretic methods, what makes it hard to give explicit estimates on the Morse index of such planes. In the same paper, they suggest also the naturally related search for unbounded free-boundary minimal surfaces in asymptotically flat manifolds with closed minimal boundaries, and the study of such surfaces in the exact Schwarzschild metric. In the present work, we obtain some progress on a basic question related to the second problem. In particular, we give evidence that under some natural hypothesis it is possible to apply min-max constructions to obtain unbounded free-boundary minimal surfaces as limits of index one minimal surfaces with partially free-boundary. Such a construction would be similar to that applied by Carlotto and De Lellis in their construction of min-max geodesics in asymptotically conical surfaces, see \cite{CarDeL}, using a suited analog three-dimensional min-max theory such as in \cite{DeLRam} and \cite{Mon-boundary}, combined with the free-boundary min-max of Li and Zhou \cite{LiZhou}. It would require also an application of the min-max index estimates, as obtained by Marques and Neves in \cite{MN-index}.

Min-max minimal surfaces have been successfully constructed in different types of non-compact spaces, such as in \cite{Mon-JDG}, \cite{c-h-m-r}, \cite{KetZho}, and \cite{ChaLio}. But the methods applied in those cases find minimal surfaces of finite area only, and possibly closed. In particular, they cannot be adapted to the present setting. 

Despite the Morse index of the plane through the origin is known, we are still not able to verify whether this surface is indeed a limit of solutions to a natural min-max problem with diverging fixed boundaries. We believe that Theorem \ref{thm-length-bound} can be applied for this purpose. In order to prove this theorem, we obtain a monotonicity formula for minimal surfaces in the Schwarzschild manifold, see Proposition \ref{prop-monotonicity}. The formula is analogous to that for minimal surfaces in Euclidean space, and relates the integral of the static potential over a truncated portion of the surface with the corresponding area of the plane through the origin. Our methods share similarities with those applied by Brendle \cite{Bre} in his classification of closed embedded constant mean-curvature surfaces in warped product manifold, and Volkmann \cite{Vol} in his study of free-boundary surfaces with square integrable mean curvature.




\section{The Morse index of the plane through the origin}\label{sect-index}

In this section we prove Theorems \ref{thm-A} and \ref{cor-max-stab}; it is divided into three subsections. In the first we introduce the Jacobi operator, give the definition of Morse index, and compute it in the case of a plane through the origin in the Schwarzschild manifold. The second subsection is devoted to a separation of variables argument, and the reduction of an eigenvalue problem to a system of Ricatti equations. In the last subsection, we study these equations and present the proofs of the results stated in the introduction.

\subsection{The second variation formula and the stability operator}

In this section we introduce the stability operator associated to the quadratic form representing the second derivative of the area functional, and compute this operator in details in the case of the free-boundary minimal surface $\Sigma_0$ of the introduction. For more details, see the discussion in section 2 of \cite{AmCaSh2}.

Let $\Sigma$ be a properly embedded surface in $(M^3, g)$ and $R>0$ be fixed. Let us consider class $C^1$ vector fields $X$ in $M$ that are tangential along $\partial M$ and supported in the domain $\{x \in M : |x|\leq R\}$. We use $\{\psi(t,\cdot)\}$ to denote the one-parameter family of diffeomorphisms associated with $X$, and use it to obtain a variation of $\Sigma$ with variational vector field $X$, i.e. we consider $\psi(t, \Sigma) = \{\psi(t, x) : x\in \Sigma\}$. The first derivative of the area functional in the direction of $X$ can be computed as
\begin{equation*}
\frac{d}{dt} \bigg|_{t=0} area_g(\psi(t, \Sigma)) = \int_{\Sigma} g(X, H) + \int_{\partial \Sigma} g(X, \nu),
\end{equation*}
where $H$ and $\nu$ denote the mean-curvature and outward pointing unit conormal vectors of $\Sigma$, respectively. It follows from this formula that free-boundary minimal surfaces are precisely the critical points of the area functional with respect to tangential variations.

Let $N$ denote a globally defined unit normal vector field along $\Sigma$. From now on, we restrict our attention to smooth variation vector fields that are normal to $\Sigma$, i.e. along the surface $X = u N$ for some class $C^{\infty}$ function $u$ in $\Sigma$. The free-boundary condition implies that $X$ is an admissible tangential variation vector field. Assuming that $\Sigma$ satisfies this condition and is minimal, the second derivative of the area can be computed as
\begin{equation*}
\frac{d^2}{dt^2}\bigg|_{t=0} area_g(\psi(t, \Sigma)) = Q(u, u),
\end{equation*}
where $Q(\cdot, \cdot)$ is the quadratic form given by
\begin{equation}\label{quadratic form}
Q(u, u) = \int_{\Sigma} (|\nabla_g u|^2_g - (Ric_g(N,N) + |A_{\Sigma}|^2_g)u^2) + \int_{\partial \Sigma} A_{\partial M}(N, N) u^2,
\end{equation}
where $A_{\Sigma}$ and $A_{\partial M}$ denote the second fundamental forms of $\Sigma$ and $\partial M$; i.e. $A_{\partial M}(V, W) = -g(\tilde{\nabla}_V \eta, W)$ where $\eta$ is outward pointing unit normal to $\partial M$, and $A_{\Sigma}(V, W) = -g(\tilde{\nabla}_V N, W)$.

\begin{defi}
Let $\Sigma(R) = \Sigma\cap \{|x|\leq R\}$. The Morse index of $\Sigma(R)$ is defined as the maximal dimension of a linear subspace $V$ of class $C^{\infty}$ functions $u$ vanishing on $\{|x|=R\}$ such that $Q(u, u)<0$, for all $u \in V\setminus \{0\}$. The Morse index of $\Sigma$ is defined as the limit of the indices of the domains $\Sigma(R)$ as $R$ tends to infinity.
\end{defi}

Integration by parts shows that the Morse index of $\Sigma(R)$ coincides with the number of negative eigenvalues (with repetitions) of the Jacobi operator
\begin{equation}\label{Jacobi-operator}
L_{\Sigma}u = \Delta_{\Sigma} u + (Ric_g(N,N)+|A_{\Sigma}|^2)u,
\end{equation}
with boundary conditions 
\begin{equation}\label{boundary-conditions}
u =0 \text{ on } \{|x|=R\}, \text{ and } \frac{\partial u}{\partial \nu} + A_{\partial M}(N,N) u =0 \text{ on } \partial \Sigma.  
\end{equation}
More precisely, if $\lambda_1 < \lambda_2 \leq \ldots \leq \lambda_k \leq  \ldots \rightarrow \infty$ is the list of eigenvalues of $L_{\Sigma}$ with boundary conditions (\ref{boundary-conditions}), with repetitions, then the number of such negative eigenvalues coincides with the index of $\Sigma(R)$.

Next, we specialize to the case of $\Sigma = \Sigma_0$, the plane through the origin in the Riemannian Schwarzschild manifold. From this point on, we will rewrite the conformal factor of the metric $g$ as 
\begin{equation}
e^{2\varphi(r)} = \bigg(1+\frac{m}{2r}\bigg)^4.
\end{equation}
We claim that the eigenvalue problem above becomes
\begin{equation}\label{eigenv-problem-1}
\Delta^{\delta}_{\Sigma_0} u + \frac{m}{r^3}e^{-\varphi(r)}u + \lambda e^{2\varphi(r)}u = 0, \text{ on } \Sigma_0(R),
\end{equation}
with boundary conditions
\begin{equation}\label{eigenvalue-problem-2}
u =0 \text{ on } \Sigma_0\cap \{|x|=R\}, \text{ and } \frac{\partial u}{\partial \nu} =0 \text{ on } \partial \Sigma_0,
\end{equation}
where $\Delta^{\delta}_{\Sigma_0} = Tr_{(\Sigma_0, \delta)} Hess_{\delta} u$ is the Euclidean Laplace operator over $\Sigma_0$. The proof of this claim is a straightforward computation using the formulas for the Laplace operator and Ricci curvature under a conformal change, see 1.158 in \cite{Bes}, and the facts that $\partial M$ and $\Sigma_0$ are totally geodesic surfaces.

\subsection{Separation of variables and reduction to a Ricatti equation}\label{sect-sep-var}

Let us start by writing the system obtained in the previous section in polar coordinates $u = u(r, \theta)$ of $\Sigma$; i.e. $(x_1, x_2, 0) = (r\cos \theta, r\sin \theta, 0)$. Using the expression of the Laplace operator in polar coordinates, we obtain
\begin{equation}\label{eq-pol-1}
u_{rr} + \frac{u_r}{r}  + \frac{u_{\theta\theta}}{r^2}  + \frac{m}{r^3}\bigg(1+\frac{m}{2r}\bigg)^{-2} u + \lambda \bigg(1+\frac{m}{2r}\bigg)^{4}u = 0, \text{ on } \Sigma_R,
\end{equation}
with boundary conditions    
\begin{eqnarray}
\label{eq-pol-2} u_r=0, \text{ on } r=m/2, \text{ and }\\
\label{eq-pol-3} u =0, \text{ on } r=R.
\end{eqnarray}

In order to analyze the system determined by equations (\ref{eq-pol-1}), (\ref{eq-pol-2}), and (\ref{eq-pol-3}), we consider the following separation of variables
\begin{equation*}
u(r, \theta) = \sum_{k\in \Z} u_k(r)\cdot e^{ik\theta}. 
\end{equation*}
Therefore, a solution to the above system must satisfy, for each $k \in \Z$,
\begin{equation}\label{eq-pol-k1}
u_{k}^{\prime\prime} + \frac{u_k^{\prime}}{r}  - \frac{k^2}{r^2} u_k + \frac{m}{r^3}\bigg(1+\frac{m}{2r}\bigg)^{-2} u_k + \lambda \bigg(1+\frac{m}{2r}\bigg)^{4}u_k = 0,
\end{equation}
with boundary conditions    
\begin{equation}
\label{eq-pol-k2} u_k^{\prime}(m/2)=0, \text{ and } u_k(R) =0.
\end{equation}

From now on, let us use 
\begin{equation}
v(r) = \sqrt{r} \cdot u_k(r).
\end{equation}
Observe that (\ref{eq-pol-k1}) and (\ref{eq-pol-k2}) imply that $v(r)$ solves
\begin{equation}\label{eq-pol-v}
v^{\prime\prime}(r) + v(r)\cdot \bigg( \frac{1}{4r^2}  - \frac{k^2}{r^2} + \frac{m}{r^3}\bigg(1+\frac{m}{2r}\bigg)^{-2} + \lambda \bigg(1+\frac{m}{2r}\bigg)^{4}\bigg) = 0,
\end{equation}
with boundary conditions
\begin{equation}
\label{eq-pol-v2} v^{\prime}\bigg(\frac{m}{2}\bigg)=m^{-1}\cdot v\bigg(\frac{m}{2}\bigg), \text{ and } v(R) =0.
\end{equation}
We claim that whenever $u_k$ is not identically zero, then the associated function $v$ has isolated zeroes only. In order to prove this claim, it suffices to verify that if $v(r)=0$, then $v^{\prime}(r)\neq 0$ by uniqueness of solutions.

Away from the zeroes of $v$, we define 
\begin{equation}
\gamma(r) = \frac{v^{\prime}(r)}{v(r)}.
\end{equation}
A straightforward computations yields the following properties of $\gamma$:
\begin{equation}\label{eq-pol-gamma1}
\gamma^{\prime}(r) + \gamma(r)^2 = -\frac{1}{4r^2}  + \frac{k^2}{r^2} - \frac{m}{r^3}\bigg(1+\frac{m}{2r}\bigg)^{-2} - \lambda \bigg(1+\frac{m}{2r}\bigg)^{4},
\end{equation}
with boundary conditions
\begin{equation}
\label{eq-pol-gamma2} \gamma\bigg(\frac{m}{2}\bigg)=m^{-1}, \text{ and } \lim_{r\rightarrow R^-} \gamma(r) = -\infty.
\end{equation}
The negative sign in the above limit can be determined from equation (\ref{eq-pol-gamma1}). Indeed, since the right-hand-side is bounded away from $r=0$, if $\gamma^2$ grows to infinity the derivative $\gamma^{\prime}$ must decrease to negative infinity. This implies, in particular, that whenever $v(r_0)=0$, we have that $\gamma(r)$ tends to $-\infty$ as $r\rightarrow r_0^{-}$, and to $+\infty$ as $r\rightarrow r_0^{+}$.


\subsection{Analysis of the Ricatti equations}\label{sect-ricatti}

In this section we study the Ricatti equation (\ref{eq-pol-gamma1}) with boundary condition (\ref{eq-pol-gamma2}). The solution to that system is defined on the closed interval $[m/2, R]$, except at finitely many points, $r=R$ included. Since we are interested in solutions with negative $\lambda$, we start with the following fact.

\begin{prop}
If $\lambda\leq 0$ and $k\neq 0$, the system composed by equations  (\ref{eq-pol-gamma1}) and (\ref{eq-pol-gamma2}) admits no solutions.
\end{prop}

\begin{proof}
We begin by estimating the right-hand-side of equation (\ref{eq-pol-gamma1}) as
\begin{equation*}
 -\frac{1}{4r^2}  + \frac{k^2}{r^2} - \frac{m}{r^3}\bigg(1+\frac{m}{2r}\bigg)^{-2} - \lambda \bigg(1+\frac{m}{2r}\bigg)^{4} \geq \frac{1}{r^2}\bigg( k^2-\frac{3}{4}\bigg).
\end{equation*}
This follows from the easily verifiable inequality 
\begin{equation*}
\frac{m}{r}\bigg(1+\frac{m}{2r}\bigg)^{-2}\leq \frac{1}{2}.
\end{equation*}

Next, we observe that the function
\begin{equation}
\psi(r) = \frac{1}{2r}\bigg( 1-\sqrt{4k^2-2}\cdot \bigg(\frac{2}{1 + (2r/m)^{\sqrt{4k^2-2}}} -1 \bigg)\bigg),
\end{equation}
considered with domain of definition $r\geq m/2$, solves
\begin{equation*}
\psi^{\prime} + \psi^2 = \frac{1}{r^2}\bigg( k^2-\frac{3}{4}\bigg),
\end{equation*}
with boundary condition $\psi(m/2) = m^{-1}$.

Therefore, it follows from the above estimate that
\begin{equation*}
\gamma^{\prime}+ \gamma^2 \geq \psi^{\prime} + \psi^2 \text{ and }  \gamma (m/2) = \psi(m/2), 
\end{equation*}
which implies that $\gamma(r) \geq \psi(r)$, for all $r$, since $\psi$ has no singularities. In particular, any function $\gamma$ satisfying equation (\ref{eq-pol-gamma1}), with $\lambda\leq 0$ and $k\neq 0$, and $\gamma(m/2) = m^{-1}$, does not have any singularities and cannot satisfy the second boundary condition expressed in (\ref{eq-pol-gamma2}).
\end{proof}

The above result has the following immediate consequences.

\begin{cor}\label{cor-radial}
\begin{enumerate}
\item[(a)] Solutions to the system composed by the equation
\begin{equation*}
L_{\Sigma} u + \lambda u =0, \text {on } \Sigma_R,
\end{equation*}
and boundary conditions
\begin{eqnarray*}
u=0, \text{ on } \Sigma \cap \{|x|=R\}, \text{ and }\\
 \frac{\partial u}{\partial \nu} = 0, \text{ on } \Sigma \cap \{|x|=m/2\},
\end{eqnarray*}
with $\lambda \leq 0$ are radial functions $u = u(r)$.

\item[(b)] Non-positive eigenvalues associated with the eigenvalue problem posed in part (a) have multiplicity one.
\end{enumerate}
\end{cor}

Part (b) follows from (a) and uniqueness of solutions to ODEs with given boundary data. Next we introduce a family of auxiliary functions that will help us on the analysis of radial eigenfunctions with negative eigenvalues.

\begin{lemma}\label{fctn-psi_c}
For each parameter $c$, the function 
\begin{equation}
\psi_c(r) = \frac{1}{2r}+\frac{4rm(4\log r + c + 8)+16r^2 - 4m^2}{r(4r^2-m^2)(4\log r +c + 8)-8r(2r+m)^2}
\end{equation}
solves equation (\ref{eq-pol-gamma1}) for $\lambda=0$ and $k=0$. The constant $c$ determines $\psi_c(m/2)$, and for $\overline{c} = -8 -4\log (m/2)$ we have $\psi_{\overline{c}}(m/2)= m^{-1}$.
\end{lemma}

The proof of this lemma is a straightforward computation that we omit here. Next, we list some further properties of the functions $\psi_c(r)$.

\begin{prop}\label{prop-R_c}
\begin{itemize}
\item[(a)] The function $\psi_c$ has a single singularity at the point $r = R_c$ such that
\begin{equation*}
(2R_c-m)(4\log R_c + 8 + c) = 8(2R_c+m).
\end{equation*}

\item[(b)] The value $R_c$ of the singularity is a strictly decreasing function of the parameter $c$ with $\lim_{c\rightarrow -\infty} R_c = +\infty$.


\end{itemize}
\end{prop}

Observe that part (a) follows from the explicit expression of $\psi_c$ given in Lemma \ref{fctn-psi_c}. Part (b) is a consequence of the fact that
\begin{equation*}
c(R):= \frac{2R+m}{2R-m} - \frac{1}{2}(\log R + 2)
\end{equation*}
is a strictly decreasing function of $R$, since this is the inverse of $R_c = R(c)$.

\begin{cor}\label{cor-1sing}
The solution $\gamma_{\lambda} = \gamma_{\lambda}(r)$ of the Ricatti equation
\begin{equation}
\gamma^{\prime} + \gamma^2 = -\frac{1}{4r^2}  - \frac{m}{r^3}\bigg(1+\frac{m}{2r}\bigg)^{-2} - \lambda \bigg(1+\frac{m}{2r}\bigg)^{4},
\end{equation}
with $\gamma_{\lambda}(m/2) = m^{-1}$, and $\lambda<0$ has at most one singularity.
\end{cor}

\begin{proof}
Following the notation of Lemma \ref{fctn-psi_c} and Proposition \ref{prop-R_c}, we observe that $\gamma_{\lambda}$ coincides with the function $\psi_{\overline{c}}$ at $r=m/2$, and 
\begin{equation*}
\gamma_{\lambda}^{\prime}+ \gamma_{\lambda}^2 > \psi_{\overline{c}}^{\prime} + \psi_{\overline{c}}^2.
\end{equation*}
This implies that $\gamma_{\lambda}(r) \geq \psi_{\overline{c}}(r)$, for every $r \in [m/2, R_{\overline{c}}]$. In particular, the function $\gamma_{\lambda}$ does not have any singularities on the interval $[m/2, R_{\overline{c}})$, since 
\begin{equation}
\lim_{r\rightarrow r_0^-} \gamma_{\lambda}(r) = -\infty \text{ and } \lim_{r\rightarrow r_0^+} \gamma_{\lambda}(r) = +\infty
\end{equation}
whenever $r_0$ is a singularity of $\gamma_{\lambda}$, as observed at the end of Section \ref{sect-sep-var}.

We claim that $R_{\overline{c}}$ is not a singularity of $\gamma_{\lambda}$ as well. Suppose, by contradiction, that this is not the case. Pick any point $p = (r_p, \gamma_{\lambda}(r_p))$ in the graph of $\gamma_{\lambda}$ different from $(m/2, m^{-1})$, and let $f = f(r)$ be the solution to (\ref{eq-pol-gamma1}) for $\lambda=0$ and $k=0$ through $p$, i.e. it solves the same differential equation as the $\psi_c$. Therefore, $f = \psi_{c}$ for some $c$, and we observe further that $c < \overline{c}$. This can be seen as an application of part (b) of Proposition \ref{prop-R_c}. Finally, we use $f$ as a barrier to obtain a contradiction. Indeed, the reasoning applied above implies that $\gamma_{\lambda}(r)\geq f(r) = \psi_c(r)$, for $r$ on $[r_p, R_c]$. Applying part (b) of the proposition again we see that $\gamma_{\lambda}$ does not develop singularity before $R_{c} > R_{\overline{c}}$. This contradicts the assumption that $R_{\overline{c}}$ is a singularity of $\gamma_{\lambda}$.

In order to conclude the proof, we observe that if $\gamma_{\lambda}$ has a singularity at $R(\lambda)$, then the facts that $R(\lambda) > R_{\overline{c}}$ and the description of the function to the right of each singularity that we gave in the previous section imply that  
\begin{equation*}
\gamma_{\lambda}(r) > \psi_{\overline{c}}(r), \text{ for all } r \text{ on } (R_{\lambda}, +\infty).
\end{equation*}
Since $\psi_{\overline{c}}$ has a unique singularity, we conclude that the same must also hold for $\gamma_{\lambda}$, and the proof is complete.
\end{proof}

We are now ready to prove that the Morse index of $\Sigma$ is one.

\begin{proof}[Proof of Theorem \ref{thm-A}]
Let $\lambda$ be a negative eigenvalue of the quadratic form associated to the second derivative of the area functional at $\Sigma_R$ with respect to vector fields that vanish at $\Sigma\cap\{|x|=R\}$ and are tangential along $\partial M$. Observe that Corollary \ref{cor-radial} implies that $\lambda$ has multiplicity one and an eigenfunction $u$ of this eigenvalue is radial, i.e. $u = u(r)$, where we used the same letter to denote the function on $\Sigma_R$ and a single variable function. 

In Section \ref{sect-sep-var} we found also equations that are satisfied by $v(r) = \sqrt{r}\cdot u(r)$ and $\gamma(r) = v^{\prime}(r) \cdot v(r)^{-1}$, see equations (\ref{eq-pol-v}), (\ref{eq-pol-v2}), (\ref{eq-pol-gamma1}), and (\ref{eq-pol-gamma2}). The zeroes of $v(r)$, which are isolated, correspond to points at which $\gamma(r)$ is not defined. Throughout this paper, these values are being referred as singularities of $\gamma(r)$. By Corollary \ref{cor-1sing}, we know that if $\gamma$ has a unique singularity, then $v$ vanishes exactly once. Since $v(R)=0$, we conclude that $u(x)>0$, for all points with $r(x)< R$. The only eigenfunction that has a well-defined sign is that associated with the first eigenvalue, which implies that $\lambda$ is the first eigenvalue and the index of $\Sigma_R$ is at most one.

In order to conclude the proof, we consider the solutions $v(r) = v_{\lambda}(r)$ to
\begin{equation*}
v^{\prime\prime}(r) + v(r)\cdot \bigg( \frac{1}{4r^2}  + \frac{m}{r^3}\bigg(1+\frac{m}{2r}\bigg)^{-2} + \lambda \bigg(1+\frac{m}{2r}\bigg)^{4}\bigg) = 0,
\end{equation*}
with boundary conditions
\begin{equation*}
v\bigg(\frac{m}{2}\bigg) = 1 \text{ and } v^{\prime}\bigg(\frac{m}{2}\bigg)=m^{-1}.
\end{equation*}
A straightforward computation verifies that
\begin{equation*}
v_{0}(r) = \sqrt{\frac{2r}{m}}\bigg( 1- \frac{2r-m}{2r+m} \cdot \log \sqrt{\frac{2r}{m}}\bigg).
\end{equation*}
Since this function takes negative values, the same happens for $v_{\lambda}(r)$ with small $|\lambda|$. This implies that the function $v_{\lambda}$ has a zero and corresponds to a radial eigenfunction of some $\Sigma_R$. 
\end{proof}

The last paragraph of our proof has Theorem \ref{cor-max-stab} as a consequence.


\section{Monotonicity formula and boundary length}

The Riemannian connection $\tilde{\nabla}$ of $g$ can be expressed in terms of the Euclidean connection $\nabla$ as follows:
\begin{equation}\label{connection-conformal}
\tilde{\nabla}_X Y = \nabla_X Y + X(\varphi)Y + Y(\varphi)X - \delta(X, Y) \nabla \varphi,
\end{equation}
where $g = e^{2\varphi}\delta$, and $\nabla \varphi$ denotes the gradient of $\varphi$.

\begin{prop}\label{prop-minimal-cone}
Let $\alpha = \alpha(s)$ be a curve parametrized by Euclidean arc-length in $\{x \in \R^3 : |x|=1\}$. The cone over $\alpha$ in the Riemannian Schwarzschild  $(M^3, g)$ will be denoted by
\begin{equation*}
C_{\alpha} = \{t \alpha(s) \in M: s\in domain(\alpha) \text{ and } t\geq m/2\}.
\end{equation*}
Then, $C_{\alpha}$ is a minimal surface if and only if $\alpha$ is a great circle.
\end{prop}

\begin{proof}
The proof of this fact follows via a straightforward computation of the mean-curvature of $C_{\alpha}$ in terms of $\alpha$ and its derivatives. Consider the orthonormal basis $\{e_1, e_2, N\}$ at $t\alpha(s)$ with respect to $g$ given by
\begin{equation*}
e_1 = e^{-\varphi(t)} \alpha(s), e_2 = e^{-\varphi(t)}\alpha^{\prime}(s), \text{ and } N = e^{-\varphi(t)} \alpha(s) \wedge \alpha^{\prime}(s).  
\end{equation*} 
The unit vector $N$ is normal to $C_{\alpha}$. The formula for the connection of $g$, see (\ref{connection-conformal}), gives us that
\begin{equation*}
\tilde{\nabla}_{e_i} N = \nabla_{e_i} N + e_i(\varphi)N,
\end{equation*}
since $N$ is orthogonal to the $e_i$, and, in particular, to the radial direction, while $\nabla \varphi$ is radial. Therefore, the mean-curvature of $C_{\alpha}$ can be computed as follows
\begin{equation*}
H = \sum_{i=1}^2 g(\tilde{\nabla}_{e_i} N, e_i) = \sum_{i=1}^2 g(\nabla_{e_i} N,  e_i) = \delta(\nabla_{\alpha^{\prime}(s)} N, \alpha^{\prime}(s)).
\end{equation*}
The derivative of $N$ in the direction of $e_1$ does not appear in the last term because it is normal to $C_{\alpha}$, since the direction of $N$ depends on $s$ only. Therefore, we conclude that
\begin{equation*}
H(t\alpha(s)) = \frac{1}{t e^{\varphi(t)}} \delta(\alpha(s)\wedge \alpha^{\prime\prime}(s), \alpha^{\prime}(s)).
\end{equation*}
Since $\alpha$ is contained in $|x|=1$ and is parametrized by arc length, we conclude that $\alpha\wedge \alpha^{\prime\prime}$ is parallel to $\alpha^{\prime}$. Therefore, we conclude that $C_{\alpha}$ is minimal if and only if $\alpha^{\prime\prime}$ is proportional to $\alpha$, which is the same as saying that $\alpha$ is a geodesic of the round two-sphere $|x|=1$.  
\end{proof}

Throughout this section, it will be convenient to use other ways to express the Schwarzschild metric. This metric coincides with
\begin{equation*}
\frac{1}{1-\frac{2m}{s}}ds\otimes ds + s^2 g_{S^2}, \text{ on } (2m, \infty)\times S^2,
\end{equation*}
where $g_{S^2}$ is the round metric on the unit two-dimensional sphere. A third way to represent that metric is as
\begin{equation*}
dr\otimes dr + h(r)^2 g_{S^2}, \text{ on } (0, \infty)\times S^2,
\end{equation*}
where $h(r)$ is the inverse function of
\begin{equation*}
r(s) = s\sqrt{1-2ms^{-1}} + m \log \bigg( \frac{1+\sqrt{1-2ms^{-1}}}{1-\sqrt{1-2ms^{-1}}}\bigg).
\end{equation*}

The variable $r = r(x)$ above represents the Schwarzschild distance to the horizon, and not the Euclidean distance; note that this notation is different from that of Section \ref{sect-index}. The function 
\begin{equation}\label{def-function-f}
f(x) = h^{\prime}(r(x))
\end{equation}
is the static potential of the Schwarzschild manifold, i.e. a non-negative function, which vanishes precisely on $\partial M$, and satisfies the equation
\begin{equation*}
Hess_{g} f - (\Delta_g f) g - f Ric_g = 0.
\end{equation*}

One can also check that the static potential satisfies
\begin{equation}\label{asymptotics-hprime}
h^{\prime}(r) = 1 - \frac{m}{r} + o\big(r^{-1}\big), \text{ as } r\rightarrow \infty.
\end{equation}
Indeed, this is a consequence of $h(r) = r + o(r)$ and
\begin{equation}\label{eqtn-derivative h}
h^{\prime}(r) = \sqrt{1- \frac{2m}{h(r)}}. 
\end{equation} 

Another important object that we apply in this work is the conformal vector field
\begin{equation*}
X = h(r) \partial_r, 
\end{equation*}
where $\partial_r$ denotes the unit length radial vector, i.e. the gradient with respect to $g$ of the function $r$. This vector field is conformal because it satisfies $\mathcal{L}_X g = 2 f g$, where $f$ is the function introduced in (\ref{def-function-f}). Moreover, since $X$ is gradient, it follows that
\begin{equation}\label{derivative-X}
\tilde{\nabla}_v X = f  v, \text{ for all } v \in TM.
\end{equation}
It might be interesting for the reader to compare this discussion on the function $f$ and the conformal vector fields in warped products with the content of section 2 of \cite{Bre}.

Let $\Sigma$ be a properly embedded minimal surface in $M$ that meets the boundary of $M$ orthogonally. For every $0<\sigma < \rho$, consider the field $W$ defined by the expression
\begin{equation*}
W(x) = \bigg( \frac{1}{h(\sigma)^2} - \frac{1}{h(\rho)^2}\bigg) X(x), 
\end{equation*}
at points $x \in M$ with $0\leq r(x)\leq \sigma$,
\begin{equation*}
W(x) = \bigg( \frac{1}{h(r)^2} - \frac{1}{h(\rho)^2}\bigg) X(x), 
\end{equation*}
for $\sigma \leq r = r(x)\leq \rho$, and $W(x)=0$, otherwise. Using (\ref{derivative-X}), we can write
\begin{equation*}
div_{\Sigma} W(x) = \bigg( \frac{1}{h(\sigma)^2} - \frac{1}{h(\rho)^2}\bigg) 2f,
\end{equation*}
for $0\leq r(x) < \sigma$, and
\begin{eqnarray*}
div_{\Sigma} W(x) &=& - \frac{2f}{h(\rho)^2} + \frac{2f}{h(r)^2} - \frac{2h^{\prime}(r)}{h(r)^3} g(\partial^{T}_r, X)\\
\nonumber & = & - \frac{2f}{h(\rho)^2} + \frac{2f}{h(r)^2} g(\partial^{\perp}_r, \partial^{\perp}_r),
\end{eqnarray*}
for $\sigma < r(x) < \rho$, where $\partial^{T}_r$ and $\partial^{\perp}_r$ denote the tangential and normal components of $\partial_r$ relative to the tangent spaces of $\Sigma$. And it is clear that the divergence of $W$ over $\Sigma$ vanishes at all points with $r(x)>\rho$.

Therefore, for almost all $0 < \sigma < \rho$, we have
\begin{equation}\label{eq-div-W}
\frac{1}{h(\rho)^2}\int_{\Sigma_{\rho}} f = \frac{1}{h(\sigma)^2}\int_{\Sigma_{\sigma}} f + \int_{\Sigma_{\rho}\setminus \Sigma_{\sigma}} \frac{f}{h^2} |\partial^{\perp}_r|^2_g  - \frac{1}{2}\int_{\Sigma} div_{\Sigma} W,
\end{equation}
where $\Sigma_a= \Sigma\cap B_{a}$, and $B_{a}$ denotes the set of points in $M$ at Schwarzschild distance to the horizon at most $a$. Since $\Sigma$ is minimal, the divergence over $\Sigma$ of the normal component of $W$ vanishes. Therefore, for almost all $0 < \sigma < \rho$, we can simplify (\ref{eq-div-W}) using the divergence theorem, and obtain
\begin{equation*}
\frac{1}{h(\rho)^2}\int_{\Sigma_{\rho}} f = \frac{1}{h(\sigma)^2}\int_{\Sigma_{\sigma}} f + \int_{\Sigma_{\rho}\setminus \Sigma_{\sigma}} \frac{f}{h^2} |\partial^{\perp}_r|^2_g  - \frac{1}{2}\int_{\partial \Sigma} g(W^T, \nu),
\end{equation*}
where $\nu$ denotes the conormal vector along $\partial \Sigma$.

Finally, since $\Sigma$ meets $\partial M$ orthogonally, it follows that $W$ is tangential to $\Sigma$, and $\nu = - \partial_r$. Therefore,  
\begin{equation*}
g(W^T, \nu) = -\bigg( \frac{1}{h(\sigma)^2} - \frac{1}{h(\rho)^2}\bigg) h(0) =  -2m\bigg( \frac{1}{h(\sigma)^2} - \frac{1}{h(\rho)^2}\bigg),
\end{equation*}
and we conclude that, for almost all $0 \leq \sigma < \rho$, we have
\begin{equation}\label{monotonicity-formula}
\frac{\mu (\Sigma\cap B_{\rho})}{h(\rho)^2}= \frac{\mu (\Sigma\cap B_{\sigma})}{h(\sigma)^2}+ \int_{\Sigma_{\rho}\setminus \Sigma_{\sigma}} \frac{f}{h^2} |\partial^{\perp}_r|^2_g  + m\bigg( \frac{1}{h(\sigma)^2} - \frac{1}{h(\rho)^2}\bigg) |\partial \Sigma |,
\end{equation}
where $ |\partial \Sigma |$ represents the boundary length of $\Sigma$, and $\mu$ is the measure defined by $\mu (A) = \int_{A} f$. An approximation argument, and the right-continuity of the integrals involved in (\ref{monotonicity-formula}) imply that it holds for all $0\leq \sigma < \rho$. We have proved the following proposition.

\begin{prop}\label{prop-monotonicity}
Let $(M, g = dr\otimes dr + h(r)^2 g_{S^2})$ denote the Riemannian Schwarzschild manifold of mass $m$, and $f = h^{\prime}(r)$ be its static potential. Let $\Sigma$ be a properly embedded minimal surface in $M$ that meets the boundary of $M$ orthogonally. Then, the ratio
\begin{equation*}
 \frac{1}{h(\rho)^2} \int_{\Sigma\cap B_{\rho}} f(x) d\Sigma(x)
\end{equation*}
is a non-decreasing function of $\rho$ in $[0, \infty)$. Moreover, if $\partial \Sigma$ is a non-trivial curve in the horizon, the ratio above is a strictly increasing quantity.
\end{prop}

Motivated by the Euclidean setting, we consider the density at infinity of a free-boundary minimal surface in the exact Schwarzschild, defined as the limit of the ratio between the area of the portion of the given surface inside the ball of radius $r$ centered at the horizon, and the area of the totally geodesic section inside the same ball, as $r$ grows to infinity. More precisely, we consider the following notion:

\begin{defi}
We define the density at infinity of a properly embedded free-boundary minimal surface $\Sigma$ in the Schwarzschild manifold by
\begin{equation*}
\Theta (\Sigma) = \lim_{\rho\rightarrow \infty} \frac{area(\Sigma\cap B_{\rho})}{area(C\cap B_{\rho})},
\end{equation*}
whenever this limit exists, where $C$ denotes the cone over a great circle contained in the horizon.
\end{defi}

A simple computation yields $area(C\cap B_r) = 2\pi \int_0^r h(\rho) d\rho$. Next, we compare the density at infinity with the limit of the monotone quantity analyzed in Proposition \ref{prop-monotonicity}.

\begin{prop}\label{prop-density-oo}
Let $\Sigma$ be a properly embedded minimal surface in $M$ that meets the boundary of $M$ orthogonally. Then,
\begin{equation*}
\lim_{\rho\rightarrow \infty}  \frac{1}{\pi h(\rho)^2} \int_{\Sigma\cap B_{\rho}} f(x) d\Sigma(x)\leq \Theta(\Sigma). 
\end{equation*}
Moreover, equality holds whenever the density at infinity of $\Sigma$ is finite.
\end{prop}

\begin{proof}
First of all, since $f(r) = h^{\prime}(r)$ we observe that (\ref{eqtn-derivative h}) implies that 
\begin{equation*}
\lim_{\rho\rightarrow \infty} \frac{1}{\pi h(\rho)^2} \int_{\Sigma\cap B_{\rho}} f(x) d\Sigma(x) \leq \liminf_{\rho\rightarrow \infty} \text{ }\frac{area(\Sigma\cap B_{\rho})}{\pi h(\rho)^2}.
\end{equation*}
On the other hand, if $\Theta(\Sigma)$ exists, we have
\begin{equation*}
\lim_{\rho\rightarrow \infty} \frac{area(\Sigma\cap B_{\rho})}{area(C\cap B_{\rho})} = \lim_{\rho\rightarrow \infty} \frac{area(\Sigma\cap B_{\rho})}{\pi h(\rho)^2},
\end{equation*}
and the inequality is proved.

Suppose now that $\Theta(\Sigma)$ is finite. In order to check the second claim in our statement, it suffices to show that 
\begin{equation}\label{eqtn-equality}
\lim_{\rho\rightarrow \infty} \frac{1}{h(\rho)^2} \int_{\Sigma\cap B_{\rho}} (f-1) = 0.
\end{equation}
Observe that there exists $c>0$ such that $|f(x)-1|\leq c\cdot r(x)^{-1}$, for all $r(x)$ large enough, and $-1 \leq f(x)-1 \leq 0$, for all $x$. These are consequences of (\ref{asymptotics-hprime}). Therefore, for $0 < r_0 < \rho < \infty$, we have
\begin{equation*}
0 \leq  \frac{1}{h(\rho)^2} \int_{\Sigma\cap B_{\rho}} (1 - f) \leq \frac{area(\Sigma\cap B_{r_0})}{h(\rho)^2} + \frac{c}{r_0}\cdot \frac{area(\Sigma\cap B_{\rho})}{h(\rho^2)}.
\end{equation*}
Note that, for fixed $r_0$, the right-hand side of this inequality converges to $c\pi \Theta(\Sigma)r_0^{-1}$ as $\rho \rightarrow \infty$. In particular, if the density at infinity of $\Sigma$ is finite, letting $r_0 \rightarrow \infty$ we conclude that (\ref{eqtn-equality}) holds.

\end{proof}

We are now ready to prove the bound on the boundary length stated in Theorem \ref{thm-length-bound}.

\begin{proof}[Proof of Theorem \ref{thm-length-bound}]
Assume, without loss of generality that the density of $\Sigma$ at infinity is finite. Letting $\sigma = 0$ and $\rho \rightarrow \infty$ in the monotonicity formula (\ref{monotonicity-formula}), and applying Proposition \ref{prop-density-oo}, we conclude that
\begin{equation*}
\Theta(\Sigma) = \frac{|\partial \Sigma|}{4\pi m} + \frac{1}{\pi}\int_{\Sigma} \frac{f}{h^2} |\partial^{\perp}_r|^2_g.
\end{equation*}
Since the integral term on the right hand side is non-negative, the inequality in the statement of the theorem is proved. If equality holds, that integral must vanish, which implies that $\partial^{\perp}_r = 0$ on $\Sigma$. In particular, $\Sigma$ must be a minimal cone. By Proposition \ref{prop-minimal-cone}, $\Sigma$ must be the cone over a great circle in the horizon, i.e. a plane through the origin.
\end{proof}

\bibliographystyle{amsbook}

\begin{thebibliography}{99}


\bibitem{AmCaSh2} Ambrozio, L., Carlotto, A., and Sharp, B., 
\textit{Compactness analysis for free boundary minimal hypersurfaces} Calc. Var. Partial Differential Equations 57 (2018), no. 1, Art. 22, 39 pp.

\bibitem{Bes} Besse, A.,
\textit{Einstein Manifolds.} Springer-Verlag, Berlin, 1987.

\bibitem{Bre} Brendle, S.,
\textit{Constant mean curvature surfaces in warped product manifolds.} Publ. Math. Inst. Hautes Études Sci. 117 (2013), 247--269.

\bibitem{Car} Carlotto, A.,
\textit{Rigidity of stable minimal hypersurfaces in asymptotically flat spaces.} Calculus of Variations and PDE, 55 (2016), no. 3, 1--20.

\bibitem{CarDeL} Carlotto, A. and De Lellis, C.,
\textit{Min-max embedded geodesic lines in asymptotically conical surfaces.} to appear in the J. Differential Geom.

\bibitem{CarChoEic} Carlotto, A., Chodosh, O., Eichmair, M.,
\textit{Effective versions of the positive mass theorem.} Invent. Math. 206 (2016), no. 3, 975--1016.

\bibitem{ChaLio} Chambers, G., and Liokumovich, Y.,
\textit{Existence of minimal hypersurfaces in complete manifolds of finite volume.} preprint, arXiv:1609.04058 [math.DG]

\bibitem{ChoKet} Chodosh, O. and Ketover, D.,
\textit{Asymptotically flat three-manifolds contain minimal planes.} to appear in Adv. Math

\bibitem{c-h-m-r}
P. Collin, L. Hauswirth, L. Mazet and H. Rosenberg, {\it Minimal surfaces in finite volume non-compact hyperbolic $3$-manifolds.}  Trans. Amer. Math. Soc. 369 (2017), no. 6, 4293--4309.

\bibitem{Dev} Devyver, B., 
\textit{Index of the critical catenoid.} arXiv:1609.02315 [math.DG]

\bibitem{DeLRam}
C. De Lellis and J. Ramic,
\textit{Min-max theory for minimal hypersurfaces with boundary.}  Ann. Inst. Fourier (Grenoble) 68 (2018), no. 5, 1909--1986. 

\bibitem{KetZho} Ketover, D., and Zhou, X., 
\textit{Entropy of closed surfaces and min-max theory.} J. Differential Geom. 110 (2018), no. 1, 31--71.

\bibitem{LiZhou} Li, M., and Zhou, X.,
\text{Min-max theory for free boundary minimal hypersurfaces I - regularity theory} preprint, arXiv:1611.02612 [math.DG]

\bibitem{MN-index}
F. Marques and A. Neves,
\textit{Morse index and multiplicity of min-max minimal hypersurfaces.} Camb. J. Math. 4 (2016), no. 4, 463--511.

\bibitem{MazRos} Mazet, L. and Rosenberg, H.,
\textit{Minimal planes in asymptotically flat three-manifolds.} arXiv:1804.05658 [math.DG]

\bibitem{Mon-JDG}
R. Montezuma,
\textit{Min-max minimal hypersurfaces in non-compact manifolds.} J. Differential Geom. 103 (2016), no. 3, 475--519.

\bibitem{Mon-boundary} R. Montezuma, 
\textit{A mountain pass theorem for minimal hypersurfaces with fixed boundary.} preprint, arXiv:1802.04757 [math.DG].

\bibitem{SchYau} Schoen, R. and Yau, S.T.,
\textit{On the proof of the positive mass conjecture in general relativity.} Comm. Math. Phys. 65, 45--76 (1979).

\bibitem{SmiZhou} Smith, G. and Zhou, D.,
\textit{The Morse index of the critical catenoid.} to appear in Geom. Dedicata

\bibitem{SmStTrZh} Smith, G.,  Stern, A., Tran, H., and Zhou, D.,
\textit{On the Morse index of higher-dimensional free boundary minimal catenoids.}  arXiv:1709.00977 [math.DG]

\bibitem{Tran} Tran, H.,
\textit{Index Characterization for Free Boundary Minimal Surfaces.} accepted to Comm. Anal. Geom

\bibitem{Vol} Volkmann, A.,
\textit{A monotonicity formula for free boundary surfaces with respect to the unit ball.}Comm. Anal. Geom. 24 (2016), no. 1, 195--221.
\end{thebibliography}

\end{document}